\documentclass[11pt,a4paper]{amsart}
\usepackage{amssymb,amsmath}
\numberwithin{equation}{section}
\vspace{8cm}
\newtheorem{theorem}{Theorem}[section]
\newtheorem{proposition}[theorem]{Proposition}
\newtheorem{corollary}[theorem]{Corollary}
\newtheorem{lemma}[theorem]{Lemma}
\newtheorem{definition}[theorem]{Definition}
\newtheorem{remark}[theorem]{Remark}

\newcommand{\R}{{\mathbb R}}

\def\lpu{\mathcal{L}_u}

\def\grad{\nabla}
\def\la{\lambda}
\newcommand{\refe}[1]{{\rm (\ref{#1})}}

\newcommand{\Mm}{\mathcal{M}^-}
\newcommand{\Mp}{\mathcal{M}^+}
\newcommand{\Mma}{\mathcal{M}^-_{\alpha, \beta}}
\newcommand{\Mpa}{\mathcal{M}^+_{\alpha, \beta}}

\begin{document}
\parindent=0pt

\title[Symmetry for viscosity solutions of fully nonlinear equations]{Symmetry and spectral properties for viscosity solutions of fully nonlinear equations}

\author[ I. Birindelli, F. Leoni \& F. Pacella]
{ Isabeau Birindelli, Fabiana Leoni \& Filomena Pacella}

\address{Dipartimento di Matematica\newline
\indent Sapienza Universit\`a  di Roma \newline
 \indent   P.le Aldo  Moro 2, I--00185 Roma, Italy.}
 \email{isabeau@mat.uniroma1.it}
\email{leoni@mat.uniroma1.it}
\email{pacella@mat.uniroma1.it}

\keywords{Fully nonlinear elliptic equations, Pucci's extremal operators, maximum principle, principal eigenvalues, symmetry of solutions, nodal eigenfunctions}
\subjclass[2010]{ 35J60}
\begin{abstract}
We study symmetry properties of viscosity solutions of fully nonlinear uniformly elliptic equations. We show that if $u$ is a viscosity solution of a rotationally invariant equation of the form $F(x,D^2u)+f(x,u)=0$, then the operator $\mathcal{L}_u=\Mp+\frac{\partial f}{\partial u}(x,u)$, where $\Mp$ is the Pucci's sup--operator, plays the role of the linearized operator at $u$. In particular, we prove that if $u$ is a solution in a radial bounded domain, if $f$ is convex in $u$ and if the principal eigenvalue of $\mathcal{L}_u$ (associated with positive eigenfunctions) in any half domain is nonnegative, then $u$ is foliated Schwarz symmetric. We  apply our symmetry results to obtain bounds on the spectrum and to deduce properties of possible nodal eigenfunctions for the operator $\Mp$.
 \end{abstract}
\maketitle

\section{Introduction}\label{intro}

This paper studies symmetry properties of solutions of fully nonlinear equations related to  
spectral properties of what,  improperly, will be  called the linearized operator.
The question we would like to answer is, which symmetry features of the domain and the operator are 
inherited by the
viscosity solutions of the homogeneous Dirichlet problem
\begin{equation}\label{eq1}
\left\{\begin{array}{c}
-F(x, D^2u)= f(x,u)\quad \hbox{ in } \Omega\, ,\\[1ex]
u=0 \quad \hbox{ on } \partial \Omega\, ,
\end{array}\right.
\end{equation}
where $\Omega\subset \R^n$, $n\geq 2$, is a  bounded domain and $F$ is a fully nonlinear uniformly 
elliptic operator. 

Starting with Alexandrov \cite{Al} and after the fundamental works of Serrin \cite{Se} and Gidas, Ni,
Nirenberg \cite{GNN} most results on symmetry of solutions rely on the moving plane method. It is 
impossible to even start mentioning all the results obtained via that method,  be they for semilinear, 
quasilinear or fully nonlinear equations. Let us just mention here the results obtained for positive solutions of fully nonlinear equations by Da Lio and Sirakov \cite{DLS}, Birindelli and Demengel \cite{BDov} and Silvestre and Sirakov \cite{SS}. 

For the purpose of this introduction, let us
emphasise its limit of application. Indeed, as it is well known by the experts, the moving plane method cannot be applied if the domain is not convex in the symmetry direction, say e.g. if $\Omega$ is an annulus, or if the nonlinear term $f(x,u)$ does not have the right monotonicity in the $x$--variable (see e.g. \cite{PR} for several counterexamples).  The moving plane method  does not apply also to sign changing solutions. 
Of course, even when $\Omega$ is  a ball and $F$ is the Laplacian, one cannot expect sign changing solutions to be radially symmetric, as it is clear exhibited by the fact that there are non radial eigenfunctions. In these cases, some other notion of symmetry is required.

In a more philosophical  understanding, the moving plane method is the tool that allows to extend  the symmetry of the principal eigenfunctions, which are the only constant sign eigenfunctions, to all positive solutions of nonlinear equations. 
It is quite natural to wonder if this analogy can be continued, i.e. 
under which conditions can one expect solutions of nonlinear equations to share the same 
symmetry of other eigenfunctions, in particular of the "second" eigenfunctions. 

Indeed, in balls or annuli, linear operators of the type $\Delta +c(x)$ do have nodal eigenfunctions, in particular the second eigenfunctions, which are symmetric though they are not  radial. For problems in non convex domains, one can imagine, and sometimes observe numerically, that even some positive solutions, like least-energy solutions, inherit only part of the symmetry of the domain,  for instance, axial symmetry. In all these cases, if the domain is rotationally symmetric, the solutions are proved to be foliated Schwarz symmetric, according to the following

\begin{definition}\label{fss} Let $B$ be a ball or an annulus in $\R^n$, $n\geq 2$. A function $u:\overline B\to \R$ is  foliated Schwarz symmetric if there exists a unit vector $p\in S^{n-1}$ such that $u(x)$ only depends on $|x|$ and $\theta= \arccos \left( \frac{x}{|x|}\cdot p\right)$, and $u$ is non increasing with respect to $\theta\in (0,\pi)$ .
\end{definition}

In other words, a foliated Schwarz symmetric function is axially symmetric with respect to the axis $\R p$ and non increasing with respect to the polar angle $\theta$. Note that a radially symmetric function is in particular foliated Schwarz symmetric with respect to any direction $p$, and for a not radial foliated Schwarz symmetric function the symmetry direction $p$ is unique. 

In the last decades, some work has been devoted to understanding  under which conditions solutions of semilinear elliptic equations are foliated Schwarz symmetric. This line of research, which strongly relies on the maximum principle,  was started by Pacella \cite{P} and then developed  by Pacella and Weth \cite{PW} and Gladiali, Pacella and Weth \cite{GPW}, see also Pacella, Ramaswamy \cite{PR} and Weth \cite{W}. In the semilinear elliptic case, by using symmetrization techniques, some results about foliated Schwarz symmetry of minimizers of associated functionals were obtained by Smets and Willem \cite{SW}, Bartsch, Weth and Willem \cite{BWW} and Brock \cite{Br}.

Let us recall some results occurring when the diffusion operator is the Laplacian,  i.e.
for solutions of
\begin{equation}\label{linear}
\left\{\begin{array}{c}
\Delta u+f(|x|,u)=0 \quad \mbox{in}\ B\, ,\\[1ex]
u=0\quad  \mbox{on}\ \partial B\, .
\end{array}
\right.
\end{equation}
Under some convexity hypotheses on $f$, it was proved in \cite{P} and \cite{PW} that a sufficient condition for the foliated Schwarz symmetry of a solution $u$ of \refe{linear} is that the first eigenvalue $\la_1(\mathcal{L}_u, B(e))$ of the linearized operator $\mathcal{L}_u=\Delta +\frac{\partial f}{\partial u}(|x|, u)$ at the solution $u$, in the half domain $B(e)=\{ x\in B\, :\, x\cdot e>0\}$ is nonnegative, for a direction $e\in S^{n-1}$. 

Furthermore, in \cite{GPW, P, PW}, it was proved that 
$\la_1(\mathcal{L}_u, B(e))\geq 0$ for some direction $e$ if $m(u)\leq n$, where $n$ is the dimension and $m(u)$ is the Morse index of the solution $u$. We recall that
the Morse index $m(u)$ is defined as the maximal dimension of a subspace of $C^{1}_0(B)$ where the quadratic form
$$Q(\varphi)=\frac{1}{2}\int_B |\grad \varphi|^2dx -\int_B \frac{\partial f}{\partial u}(|x|, u)\varphi^2dx $$
is negative definite, or,  equivalently,   as the number of negative eigenvalues 
of the linearized operator ${\mathcal L_u}$.

In this line of thought,  the first question  is: what plays the role of the linearized operator for the fully nonlinear problem \refe{eq1}?

In the whole paper we will suppose that $F$ is uniformly elliptic (see condition \refe{ue}) and 
Lipschitz continuous in $x$ (see condition \refe{lip}). Let us recall that uniform ellipticity is equivalent to
$$
\Mma (M-N)\leq F(x,M)-F(x,N)\leq \Mpa (M -N)\quad \forall\, x\in \Omega\, ,\ M, N\in {\mathcal S}_n\, ,
$$
where $\Mma$ and $\Mpa$ are the Pucci's extremal operators with ellipticity constants $0<\alpha\leq\beta$
(for a precise definition,  see Section 2) and ${\mathcal S}_n$ is the set of $n\times n$ symmetric matrices.
This, in order, will imply 
(see Lemma \ref{diffe})  that any "derivative" $v$ of $u$ will satisfy 
$$
-\Mpa (D^2 v)\leq \frac{\partial f}{\partial u}(x, u)\, v\quad \hbox{ in } \Omega\, ,
$$
$v$ being only  a viscosity subsolution, and not a solution as in the semilinear case.
This suggests to define as "linearized" operator the fully nonlinear operator
 $$
\mathcal{L}_u(v) \, :\, = \Mpa(D^2v) +\frac{\partial f}{\partial u}(x,u)\, v\, .
$$
In this framework, it doesn't seem possible to associate a quadratic form to $\mathcal{L}_u$ in 
order to define a notion of Morse index. On the other hand, the use of the eigenvalues to define it would require the knowledge of the spectrum of the 
operator $\mathcal{L}_u$. 
In analogy with the linear case, see \cite{BNV, BD, BEQ, II}, in a domain $D$ one may define, through the maximum principle, the principal eigenvalues $\lambda_1^+=\lambda_1^+(\mathcal{L}_u , D)$ and $\lambda_1^-=\lambda_1^-(\mathcal{L}_u , D)$.
 Associated with these values,   there are principal eigenfunctions $\phi_1^{\pm}\in C(\overline D)\cap C^2(D)$, defined up to positive   constant multiples, which satisfy respectively
$$
\left\{
\begin{array}{c}
-\mathcal{L}_u  [\phi_1^+]= \lambda_1^+\, \phi_1^+\quad \hbox{in } D\\[2ex]
\phi_1^+ >0 \ \hbox{in } D\, ,\ \phi_1^+=0\ \hbox{on } \partial D
\end{array}\right.
\quad \mbox{and}\quad 
\left\{
\begin{array}{c}
-\mathcal{L}_u  [\phi_1^-]= \lambda_1^-\, \phi_1^-\quad \hbox{in } D\\[2ex]
\phi_1^- <0 \ \hbox{in } D\, ,\ \phi_1^-=0\ \hbox{on } \partial D.
\end{array}\right.
$$
However, besides the principal eigenvalues and their corresponding eigenfunctions, 
not much is known about other eigenvalues.  A completeness result of a spectral basis  
is known only for radial eigenfunctions, see Ikoma and Ishii \cite{II},  but the sign of radial eigenvalues is not relevant for the foliated Schwarz symmetry.

Nonetheless we obtain some symmetry results, by analyzing  the sign of the principal eigenvalue of $\mathcal{L}_u$ in half domains.

In order to describe our results, let us   introduce a few notations that will always be valid in the sequel.
$B$ will always denote a bounded radial domain, that is a ball or an annulus centered at the origin. 
For any unit vector $e\in S^{n-1}$, we further denote by  
$H(e)=\{ x\in \R^n\, :\, x\cdot e=0\}$ the hyperplane orthogonal to $e$ and by 
$B(e)=\{ x\in B\, :\, x\cdot e>0\}$ the open half domain on the side of $H(e)$ which contains $e$. 
Moreover, we indicate with $\sigma_e : \R^n\to \R^n$ the reflection with respect to $H(e)$, that is 
the map $\sigma_e(x)=x- 2 (x\cdot e)e$.  Accordingly, for any domain $\Omega$, we will set $\Omega(e)= \{ x\in \Omega\, :\, x\cdot e>0\}$.

A first simple, but useful,   result that we get is a sufficient condition for the symmetry of a viscosity solution $u$ of \refe{eq1} in a domain $\Omega$ symmetric with respect to a certain hyperplane $H(e)$. More precisely, we will show that if $f(x,s)$ is convex in the $s$--variable, then the positivity of the principal eigenvalues $\la_1^+(\mathcal{L}_u, \Omega(\pm e))$ in both domains $\Omega(\pm e)$ implies that $u(x)=u(\sigma_e(x))$ for all $x\in \Omega$, see Proposition \ref{prop1}.

Next, our main result, concerning the foliated Schwarz symmetry of viscosity solutions of \refe{eq1}, is
\begin{theorem}\label{lam0}Suppose that $F$ is invariant with respect to any reflection $\sigma_e$ and  by rotations. 
Let $u$ be a viscosity solution of problem \refe{eq1}, with $\Omega=B$ and $f(x,\cdot)=f(|x|,\cdot)$  convex in $\R$. If there exists $e\in S^{n-1}$ such that  
$$\lambda_1^+(\mathcal{L}_u, B(e))\geq 0,$$ 
then $u$ is foliated Schwarz symmetric.
\end{theorem}

So, under the convexity assumption on $f$, the knowledge of the sign of the principal eigenvalue $\lambda_1^+(\mathcal{L}_u, B(e))$ in one cap $B(e)$ only is sufficient for the foliated Schwarz symmetry of a solution $u$ of \refe{eq1}, for any fully nonlinear uniformly elliptic operator $F$ with ellipticity constants $0<\alpha\leq \beta$.

In any bounded domain $\Omega$, we further define
\begin{equation}\label{mu}
\mu^+_2(\lpu,\Omega)=\inf_{D\subset \Omega}\max\left\{\lambda_1^+(\lpu,D),\lambda_1^+(\lpu,\Omega\setminus \overline D)\right\}
\end{equation}
where the infimum is taken on all subdomains $D$ contained in $\Omega$.

Then, for $\Omega=B$, we immediately obtain 
$$
\mu^+_2(\lpu,B)\geq 0\Rightarrow \forall e\in S^{n-1}\ \mbox{ either}\ \lambda_1^+(\lpu,B(e))\geq 0 
\ \mbox{ or}\ \lambda_1^+(\lpu,B(-e))\geq 0
$$
so that,  by applying Theorem \ref{lam0}, the following corollary holds.
\begin{corollary}\label{sec} Under the assumptions of Theorem \ref{lam0}, 
if $u$ is a viscosity solution of \refe{eq1} and $\mu^+_2(\lpu,B)\geq 0$, then $u$ is foliated Schwarz symmetric.
\end{corollary}
 In the semilinear case, i.e. when $\lpu= \Delta +f'(|x|,u)$,  $\mu_2^+$ is just the second eigenvalue of $\lpu$, and therefore   the condition $\mu^+_2\geq 0$ is equivalent to require  that the Morse index of $u$ is less than or equal to one,  which is the condition used in \cite{P}  to obtain the foliated Schwarz symmetry. It turns out that, in the currently considered fully nonlinear case, $\mu^+_2$ is not an eigenvalue for $\lpu$ in $B$, as shown in  Proposition \ref{secn}.

We further observe that for the laplacian the first eigenvalue in the half domains $B(e)$ is the second eigenvalue in $B$. Then, it is natural to investigate if, also in the fully nonlinear framework, $\la_1^+(\Mpa, B(e))$ is a nodal eigenvalue for $\Mpa$ in $B$. Here and in the sequel by nodal eigenvalue we mean an eigenvalue associated with sign changing eigenfunctions. First we will prove   that $\Mpa$ cannot have nodal eigenvalues below $\la_1^+(\Mpa, B(e))$, and then we will show several properties that a nodal eigenfunction corresponding to $\la_1^+(\Mpa, B(e))$ should have.

Finally, we prove an interesting connection between   the sign of the principal eigenvalue of $\mathcal{L}_u$ in half domains  and   the nodal set 
$\mathcal{N}(u)$, i.e. the closure of the zero set of $u$. Namely, we prove that if $u$ is a sign changing viscosity solution of \refe{eq1} with $f$ independent of $x$ and with $u$ and $F$ symmetric with respect to an hyperplane $H(e)$, then the non negativity of the eigenvalue $\la_1^+(\mathcal{L}_u, \Omega(\pm e))$ implies that the nodal set $\mathcal{N}(u)$ intersects the boundary of $\Omega$, see Proposition \ref{prop2}.  As a consequence of the above result, we obtain that for any radial sign changing solution $u$  one has
 $\la_1^+(\mathcal{L}_u, B(e))< 0$ for any direction $e$, see Corollary \ref{rad-}.
In $\R^2$, the above result can be extended to a  larger class of domains,  i.e. domains which are symmetric with respect to 
two orthogonal directions and convex in those directions. Interestingly, besides the ball, the only two 
dimensional domains for which the eigenvalues of  $\Mpa$ are known explicitly have these symmetry, see \cite{BL}.

Let us finally point out that our symmetry results apply to viscosity solutions, and not only to classical solutions, of \refe{eq1}. This is essential in view of the fact that, in general, axially symmetric viscosity solutions of fully nonlinear equations may not be of class $C^2$,  as proved by Nadirashvili and Vl\u{a}du\c{t} \cite{NV}.

The paper is organized in the following way. The hypotheses and some preliminaries are recalled in the next section. In the third section we prove some  symmetry results. Foliated Schwarz symmetry is then studied in the fourth section. Finally, in the last section, we give some applications, in particular to  the study of   spectral properties.

\section{Preliminaries on fully nonlinear elliptic equations}\label{sec1}

We assume that
$F:\Omega\times \mathcal{S}_n\to \R$ is a continuous function, with $\mathcal{S}_n$ denoting the set of symmetric $n\times n$ matrices equipped with the usual partial ordering
$$
M\geq N \Longleftrightarrow M-N\geq 0 \Longleftrightarrow (M-N)\xi \cdot \xi \geq 0\quad \forall\, \xi\in \R^n\, .
$$
We will always assume that $F$ is  uniformly elliptic, that is
\begin{equation}\label{ue}
\alpha \, {\rm tr}(P)\leq F(x, M+P)-F(x, M)\leq \beta\, {\rm tr}(P)\, ,\quad \forall\, x\in \Omega\, ,\ M, P\in {\mathcal S}_n,\ P\geq 0\, ,
\end{equation}
for positive constants $0<\alpha\leq \beta$.  Let us recall that condition \refe{ue} is equivalent to
$$
\Mma (M-N)\leq F(x,M)-F(x,N)\leq \Mpa (M -N)\quad \forall\, x\in \Omega\, ,\ M, N\in {\mathcal S}_n\, ,
$$
where $\Mma$ and $\Mpa$ are the Pucci's extremal operators defined respectively as
$$
\begin{array}{c}
\displaystyle \Mma (M)= \inf_{A\in \mathcal{A}_{\alpha, \beta}} {\rm tr} (AM)= \alpha \sum_{\mu_i>0} \mu_i + \beta\sum_{\mu_i<0} \mu_i\\[1ex]
\displaystyle
\Mpa (M)= \sup_{A\in \mathcal{A}_{\alpha, \beta}} {\rm tr} (AM)= \beta\sum_{\mu_i>0} \mu_i + \alpha \sum_{\mu_i<0} \mu_i
\end{array}
$$
where $\mathcal{A}_{\alpha,\beta}=\{ A\in \mathcal{S}_n\, :\, \alpha\, I_n\leq A\leq \beta\, I_n\}$, $I_n$ being the unit matrix in $\mathcal{S}_n$, and $\mu_1, \ldots ,\mu_n$ being the eigenvalues of the matrix $M\in \mathcal{S}_n$.
 Thus, Pucci's extremal operators act as barriers  for  the whole class of  uniformly elliptic operators, and for a detailed analysis of the crucial role they play in the regularity theory for elliptic equations we refer to \cite{CaCa}. Clearly, $F(x, M)=\Mpa(M)$ or $F(x, M)=\Mma(M)$ are special cases  which can be considered as  our  model cases; in particular since they are invariant with respect to rotation and reflection.
From now on, we intend the ellipticity constants $\beta\geq \alpha$ fixed once and for all, and we will write just $\Mm$ and $\Mp$ for the Pucci's operators with ellipticity constants $\alpha$ and $\beta$.

As for the dependence on $x$ of $F$, we assume  Lipschitz continuity,  i.e. the existence of  $L>0$ such that, for all $x, y \in \Omega$ and $M\in {\mathcal S}_n$,
\begin{equation}\label{lip}
|F(x,M)-F(y,M)|\leq L\, \|M\||x-y|\, . 
\end{equation}
On the zero order nonlinearity $f$ we assume that it is of class $C^1$ on $\Omega\times \R$. 

By a solution of the Dirichlet problem \refe{eq1}, we always mean a viscosity solution $u\in C(\overline B)$. For the reader's convenience, we recall that a solution in the viscosity sense is both a viscosity subsolution  and a viscosity supersolution, as defined below.
\begin{definition} A viscosity subsolution (supersolution) of problem \refe{eq1} is an upper (lower) semicontinuous function in $\overline \Omega$ such that $u\leq(\geq) 0$ on $\partial \Omega$ and for any $x_0\in \Omega$ and $\phi\in C^2(\Omega)$ such that  $u(x_0)=\phi(x_0)$ and $u(x)\leq (\geq) \phi (x)$ for $x\in \Omega$, one has
$$
-F(x_0, D^2\phi (x_0))\leq(\geq) f(x_0, u(x_0))\,.
$$
\end{definition}
We refer to \cite{CaCa, CIL} the reader not familiar with the viscosity solutions theory for fully nonlinear equations.
In the following, all the differential inequalities we are going to consider are always understood in the viscosity sense.

Let us further recall that in the current assumptions, by standard elliptic regularity theory (see \cite{CaCa, S}), any viscosity solution $u$ of problem \refe{eq1} is of class $C^1(\overline \Omega)$, provided that $\Omega$ is of class $C^1$. As far as existence of solutions is concerned, we refer to \cite{FQ, QS}.

 In the subsequent symmetry results a crucial role will be played by the principal eigenvalues of linear perturbations of Pucci's operators. In particular, given a Lipschitz domain $D\subset \R^n$ and a  function $c\in C(\overline D)$, let us consider the uniformly  elliptic operator
$$
\mathcal{L} =\Mp +c(x)\, .
$$
In analogy with the linear elliptic case,  see \cite{BNV}, one may define
$$
\lambda_1^+(\mathcal{L} , D) \, := \sup \{ \lambda \in \R\, :\, \exists\, \varphi \in C(\overline D)\, ,\ \varphi>0 \hbox{ in } D\, ,\ -\mathcal{L} [\varphi]\geq \lambda\, \varphi \hbox{ in } D\}
$$
and
$$
\lambda_1^-(\mathcal{L} , D) \, := \sup \{ \lambda \in \R\, :\, \exists\, \varphi \in C(\overline D)\, ,\ \varphi<0 \hbox{ in } D\, ,\ -\mathcal{L} [\varphi]\leq \lambda\, \varphi \hbox{ in } D\}\,.
$$
As it is well known, see \cite{BD, BEQ, II}, associated with these values, called principal eigenvalues,  there are principal eigenfunctions $\phi_1^{\pm}\in C(\overline D)\cap C^2(D)$, defined up to positive   constant multiples, which satisfy respectively
\begin{equation}\label{eigenf+}
\left\{
\begin{array}{c}
-\mathcal{L}  [\phi_1^+]= \lambda_1^+(\mathcal{L} , D)\, \phi_1^+\quad \hbox{in } D\\[2ex]
\phi_1^+ >0 \ \hbox{in } D\, ,\ \phi_1^+=0\ \hbox{on } \partial D
\end{array}\right.
\end{equation}
\begin{equation}\label{eigenf-}
\left\{
\begin{array}{c}
-\mathcal{L}  [\phi_1^-]= \lambda_1^-(\mathcal{L} , D)\, \phi_1^-\quad \hbox{in } D\\[2ex]
\phi_1^- <0 \ \hbox{in } D\, ,\ \phi_1^-=0\ \hbox{on } \partial D.
\end{array}\right.
\end{equation}
When no ambiguities arise, the eigenvalues will be denoted by $\lambda_1^{+}$ or $\lambda_1^-$, and in certain cases we will only specify  either the domain $D$ or the choice of the operator.

A few known properties concerning these eigenvalues are used in the paper, we list them here.
\begin{proposition}\label{prl} With the above notations,  the following properties hold:
\begin{itemize}
\item[(i)] If $D_1\subset D_2$ and $D_1\neq D_2$, then $\la_1^\pm(D_1)>\la_1^\pm(D_2)\, .$
\item[(ii)] For a sequence of domains $\{D_k\}$ such that $D_k\subset D_{k+1}$, then 
$$\lim_{k\to +\infty}\la_1^\pm(D_k)=\la_1^\pm(\cup_k D_k)\, .$$
\item[(iii)] If $\alpha<\beta$ then $\la_1^+<\la_1^-\, .$
\item [(iv)] If $\la\neq \la_1^\pm$ is an  eigenvalue then every corresponding eigenfunction changes sign.
\item[(v)] $\la_1^+(D)>0 \ (\la_1^-(D)>0)$ if and only if the maximum (minimum) principle holds for $\mathcal{L}$ in $D\, .$
\smallskip

\item[(vi)] $\la_1^\pm(D)\to +\infty$ as ${\rm meas}(D)\to 0\, .$
\end{itemize}
\end{proposition}
Let us recall that  the operator $\mathcal{L}$ satisfies the maximum (minimum) principle in $\Omega$ if for every  function $u$ upper (lower) semicontinuous in $\overline \Omega$ satisfying $-\mathcal{L}  [u] \leq 0$ in $\Omega$ and $u\leq 0$ on $\partial \Omega$ (resp. $-\mathcal{L} [u]\geq 0$ in $\Omega$ and $u\geq 0$ on $\partial \Omega$) one has $u\leq 0$ in $\overline \Omega$ ($u\geq 0$ in $\overline \Omega$). 

Finally we recall that the principal eigenfunctions are the only positive (negative) supersolutions  of \refe{eigenf+}(subsolutions of \refe{eigenf-}) and that the following  proposition, which will be used frequently in the sequel, holds true.  

\begin{proposition} \label{key} Assume that there exists $u$ lower semicontinuous and positive such that
$$ -\mathcal{L}[u]\geq 0\ \ \mbox{in}\ \Omega\, . $$
If there exists  a function $v$ upper semicontinuous in $\overline \Omega$  satisfying $-\mathcal{L}[v]\leq 0$ in $\Omega$,  $v\leq 0$ on $\partial \Omega$, and  such that $v(\hat x)>0$
for some $\hat x\in\Omega$, 
then, for some $t>0$,
$$v\equiv t\, u\ \mbox{and}\  -\mathcal{L}[u]= 0.$$
\end{proposition}
For the proof  we refer to \cite{BNV, Pat}.

\section{First symmetry results}
 Here and in the sequel we set
$f'(x,s)= \frac{\partial f}{\partial s} (x,s)$ and we use the notations fixed in the Introduction. Moreover, 
for any two linearly independent unit vectors $e, e'\in S^{n-1}$, we denote by $\Pi(e,e')$ the plane 
spanned by $e$ and $e'$, and by $\theta_{e,e'}$ any polar angle coordinate in $\Pi(e,e')$. If $u:B\to \R$ is a differentiable function, we set $u_{\theta_{e,e'}}$ to indicate the partial derivative of $u$ with respect to $\theta_{e,e'}$, defined as zero at the origin if $B$ is a ball.

The following technical lemma is the starting point of all our symmetry results.

\begin{lemma}\label{diffe} Assume that $F$ satisfies \refe{ue} and \refe{lip} and let $u\in C(\overline \Omega)\cap C^1(\Omega)$ be a viscosity solution of \refe{eq1}
\begin{itemize}
\item[(i)]  Assume that $\Omega$ is symmetric with respect to the hyperplane $H(e)$,  $F$  is invariant with respect to the reflection $\sigma_e$, i.e.
\begin{equation}\label{sime}
F\left(\sigma_e(x), (I_n-2e\otimes e)M(I_n-2e\otimes e)\right)=F(x,M)\quad \forall\, x\in \Omega\, ,\ M\in \mathcal{S}_n\, ,
\end{equation}
and that $f$ satisfies
\begin{equation}\label{f}
f(\sigma_e(x), s)=f(x, s)\, ,\quad f(x, \cdot) \hbox{ is convex in $\R$}\, , \quad \forall\, x\in \Omega\, ,\ s\in \R\, .
\end{equation}
Then, the function  $w=u-u\circ \sigma_e$ satisfies 
$$
-\Mp (D^2 w)\leq f'(x, u)\, w\quad \hbox{ in } \Omega
$$
in the viscosity sense. 
Moreover, if $f(x,\cdot)$ is strictly convex,  then either  $w\equiv 0$ or $w$ is a strict subsolution.
\item[(ii)] Assume that $\Omega =B$ is a bounded radial domain,   $F$  is invariant by rotations, i.e. for every orthogonal matrix $O$ one has
\begin{equation}\label{ri}
F(O^t x, O^t MO)= F(x, M) \quad \forall\, x\in B\, ,\ M\in \mathcal{S}_n\, ,
\end{equation}
and that $f$ is radially symmetric in $x$. 
Then, for any pair of linearly independent unit vectors $e,\ e'\in S^{n-1}$, the functions $u_{\theta_{e,e'}}$ and $-u_{\theta_{e,e'}}$ both satisfy 
$$
\begin{array}{c}
-\Mp (D^2u_{\theta_{e,e'}})\leq f'(|x|,u)\, u_{\theta_{e,e'}}\quad \hbox{ in } B\\[1ex]
 -\Mp (D^2(-u_{\theta_{e,e'}}))\leq f'(|x|,u)\, (-u_{\theta_{e,e'}})\quad \hbox{ in } B
\end{array}
$$
in the viscosity sense.
\item[(iii)] Assume that $f$ does not depend on $x$. Then, for every $1\leq i\leq n$ both the partial derivative $u_i=\frac{\partial u}{\partial x_i}$ and $-u_i$ satisfy 
$$
\begin{array}{c}
-\Mp \left( D^2u_i\right)\leq f'(u)\, u_i\quad \hbox{ in } \Omega\\[1ex]
-\Mp \left( D^2(-u_i)\right)\leq f'(u)\,(- u_i) \quad \hbox{ in } \Omega
\end{array}
$$
in the viscosity sense.
\end{itemize}
\end{lemma}

\proof (i) Let $u\in C(\overline \Omega)$ be a viscosity solution of \refe{eq1}. By the invariance of the equation with respect to the reflection $\sigma_e$,   $u\circ \sigma_e$ is also a viscosity solution of \refe{eq1}. Then, the difference
$w=u-u\circ \sigma_e$ is a viscosity subsolution of
\begin{equation}\label{subs}
-\Mp (D^2 w)\leq f(x, u)-f(x, u\circ \sigma_e)\quad \hbox{ in } \Omega\, .
\end{equation}
If $u$ and $u\circ \sigma_e$ are classical solutions of \refe{eq1}, then \refe{subs} is an immediate consequence of the uniform ellipticity of $F$. In the general case, this follows from assumptions \refe{ue} and \refe{lip} by means of the standard regularization procedure by sup/inf--convolution, in the spirit of Theorem 5.3 of \cite{CaCa}. For a detailed proof we refer to the proof of Proposition 2.1 in \cite{DLS}. By \refe{subs} and the convexity of $f(x, \cdot)$, we immediately get the conclusion.

(ii) Let us fix $e, e', \theta_{e,e'}$ as in the statement. 
We aim  at  "differentiating" with respect to $\theta_{e,e'}$ the equation satisfied by $u$. Let us fix $\theta_0\in \R$, and let us denote by $\mathcal{R}_0:\R^n\to \R^n$ the rotation that maps any point $x$ having cylindrical coordinates $(r,\theta,\eta)$ with respect to the plane $\Pi(e,e')$ into the point $\mathcal{R}_0(x)$ with cylindrical coordinates $(r, \theta+\theta_0,\eta)$. Let us further set $u_0(x)=u(\mathcal{R}_0(x))$. Then, by the rotational invariance of $F$ and $f$, we have that both $u$ and $u_0$ satisfy
$$
-F(x, D^2 u)=f(|x|, u)\, ,\quad -F(x, D^2u_0)=f(|x|, u_0)\quad \hbox{ in } B\, .
$$
By uniform ellipticity, arguing as in the proof of (i), we get that the difference $u_0-u$ satisfies, in the viscosity sense,
$$
-\Mp (D^2(u_0-u)) \leq f(|x|, u_0)- f(|x|, u)\quad \hbox{ in } B\, .
$$
Next, by the  homogeneity properties of $\Mp$, we also have that for all $\theta_0>0$
$$
-\Mp \left( D^2\left( \frac{u_0-u}{\theta_0}\right) \right)\leq \frac{f(|x|, u_0)- f(|x|, u)}{\theta_0}\quad \hbox{ in } B\, ,
$$
whereas, for all $\theta_0<0$,
$$
\Mp \left( -D^2\left( \frac{u_0-u}{\theta_0}\right) \right)\geq \frac{f(|x|, u_0)- f(|x|, u)}{\theta_0}\quad \hbox{ in } B\, .
$$
By letting $\theta_0\to 0^\pm$, and using the stability properties of viscosity subsolutions and the fact that $\frac{u_0-u}{\theta_0} \to u_{\theta_{e,e'}}$ locally uniformly in $B$, we finally obtain both
$$
\Mp (D^2u_{\theta_{e,e'}}) + f' (|x|, u)\, u_{\theta_{e,e'}}\geq 0 \quad \hbox{ in }B\, ,
$$
and
$$
\Mp (D^2(-u_{\theta_{e,e'}})) + f' (|x|, u)\, (-u_{\theta_{e,e'}})\geq 0 \quad \hbox{ in }B\, ,
$$
in the viscosity sense.

(iii) The proof runs as for (ii).

\hfill$\Box$

\begin{remark}\label{concave}
{\rm In statement (i), if $f(x,\cdot)$ is assumed to be concave, then one has
$$
-\Mm (D^2 w)\geq f'(x, u)\, w\quad \hbox{ in } \Omega\, .
$$}
\end{remark}

We are now ready to prove our first symmetry result for viscosity solutions. If $u$ is a  viscosity solution of \refe{eq1},  we denote by $\mathcal{L}_u$ the  "linearized" fully nonlinear operator
$$
\mathcal{L}_u \, :\, = \Mp +f'(x,u)\, ,
$$
and by $\la_1^\pm(\mathcal{L}_u, \Omega(e))$,  $\la_1^\pm(\mathcal{L}_u, \Omega(-e))$ the principal eigenvalues of $\mathcal{L}_u$ in the domains $\Omega(\pm e)=\Omega \cap \{ x\cdot (\pm e)>0\}$.

\begin{proposition}\label{prop1} Assume that $\Omega$ is symmetric with respect to the hyperplane $H(e)$, $F$ satisfies \refe{ue}, \refe{lip} and \refe{sime} and that $f$ satisfies \refe{f}. Let $u$ be a viscosity solution of \refe{eq1} and assume further that either 
\begin{itemize}
\item[(i)]
$\la_1^+(\mathcal{L}_u, \Omega(\pm e))>0$
\end{itemize}
 or
 \begin{itemize}
\item[(ii)]
$\la_1^+(\mathcal{L}_u, \Omega(\pm e))\geq 0$ and  $f$ is strictly convex.
\end{itemize}
Then, $u$ is symmetric with respect to the hyperplane $H(e)$.
\end{proposition}
\proof Let us set $w=u-u\circ \sigma_e$ and observe that, by definition, $w$ is antisymmetric with respect to $H(e)$ and satisfies $w=0$ on $\partial \Omega(\pm e)$. By Lemma \ref{diffe},   $w$ is a viscosity subsolution of
$$
\left\{
\begin{array}{c}
-\Mp (D^2w) \leq f'(x,u)\, w \quad \hbox{ in } \Omega(\pm e)\\[1ex]
w=0 \quad \hbox{ on } \partial \Omega(\pm e)
\end{array}\right.
$$
If $\la_1^+(\mathcal{L}_u, \Omega (\pm e))>0$, then by the maximum principle  both $w\leq 0$ in $\Omega(e)$ and $w\leq 0$ in $\Omega(-e)$, so that, by antisymmetry, $w\equiv 0$ in $\Omega$. If one of the eigenvalues is zero, say $\la_1^+(\mathcal{L}_u, \Omega (e))=0$ and $\la_1^+(\mathcal{L}_u, \Omega(-e))>0$ and, by contradiction, $w\not\equiv 0$, then $w<0$ in $\Omega(-e)$ by the strong maximum principle. Therefore, $w>0$ in $\Omega(e)$ and by Proposition \ref{key} $w$ satisfies
$$
-\Mp (D^2w)= f'(x,u)\, w\quad \hbox{ in } \Omega(e)\, ,
$$
a contradiction to the strict convexity of $f$ by Lemma \ref{diffe} (i).
Analogously, if both $\la_1^+(\mathcal{L}_u, \Omega(\pm e))=0$, then, either $w\leq 0$ in $\Omega(\pm e)$, and then again $w\equiv 0$ in $\Omega$, or, otherwise, $w$ is a solution either in $\Omega(e)$ or in $\Omega(-e)$, in contrast with the strict convexity of $f$.

\hfill$\Box$

\begin{remark}{\rm Since $\la^-(\Mm+f^\prime(x,u),D)=\la^+(\Mp+f^\prime(x,u),D)$ for any domain $D$, by Remark \ref{concave} the same conclusion of Proposition \ref{prop1} holds if $f$ is concave.}
\end{remark}

In the remaining part of this section we will exhibit a  sufficient condition for the eigenvalue 
$\la_1^+(\mathcal{L}_u, \Omega(e))$ to be negative when $u$ is a sign changing viscosity solution symmetric 
with respect to the hyperplane  $H(e)$. 

Let us fix, for simplicity, $e=e_1=(1,0,\ldots ,0)\in S^{n-1}$ and  let  $\Omega$ be  a smooth bounded domain symmetric with respect to $H(e_1)$ and convex in the $x_1$--direction, i.e. for any two points in $\Omega$ having the same $x_1$--coordinate, the segment joining them is also contained in $\Omega$. We are going to consider a viscosity solution $u\in C^1(\overline \Omega)$ of the problem
\begin{equation}\label{eq2}
\left\{\begin{array}{c}
-F(x, D^2u)= f(u)\quad \hbox{ in } \Omega\, ,\\[1ex]
u=0 \quad \hbox{ on } \partial \Omega
\end{array}\right.
\end{equation}
Note that in \refe{eq2} $f$ does not depend on $x$. 
We recall that the nodal set $\mathcal{N} (u)$ of a solution $u$ of \refe{eq2} is defined as
$$
\mathcal{N}(u) \, := \overline{\left\{ x\in \Omega\, :\, u(x)=0\right\}}\, .
$$
\begin{proposition}\label{prop2}
Let $u\in C^1(\overline \Omega)$ be a sign changing   viscosity solution of \refe{eq2}, with $F$ satisfying \refe{ue}, \refe{lip} and \refe{sime} with $e=e_1$ and assume that  $u$ is even with respect to $x_1$. Then
$$ \la_1^+(\mathcal{L}_u, \Omega(\pm e_1))\geq 0\quad  \Longrightarrow \quad \mathcal{N}(u)\cap \partial \Omega\neq \emptyset \, . 
$$
\end{proposition}

\proof We follow the argument used in \cite{AP} for  semilinear equations, and we prove the equivalent implication
$$ \mathcal{N}(u)\cap \partial \Omega= \emptyset \quad \Longrightarrow  \quad \la_1^+(\mathcal{L}_u, \Omega(\pm e_1))< 0  \, . 
$$

Let us consider the continuous function $u_1=\frac{\partial u}{\partial x_1}$ in the cap $\Omega(e_1)$. We notice that $u_1=0$ on $H(e_1)\cap \overline \Omega$ and that $u_1$ does not change sign on $\partial \Omega \cap \partial \Omega(e_1)$. Indeed, if there were points $Q_1\,,\ Q_2\in \partial \Omega \cap \partial \Omega(e_1)$ such that $u_1(Q_1)>0$ and $u_1(Q_2)<0$, then, since $u=0$ on $\partial \Omega$, there would exist a sequence of points $\{ x_k\}\subset \Omega$ such that $u (x_k)=0$ for every $k$ and dist$(x_k, \partial \Omega)\to 0$. This would be a contradiction to the hypothesis $\mathcal{N}(u)\cap \partial \Omega=\emptyset$. Hence, either $u_1\leq 0$ or $u_1\geq 0$ on $\partial \Omega \cap \partial \Omega(e_1)$. We can assume without loss of generality to be in the first case, since otherwise we can  consider, by the symmetry of $u$,  the opposite cap $\Omega(-e_1)$. 
We further observe that, by Lemma \ref{diffe} (iii), $u_1$ satisfies in the viscosity sense
\begin{equation}\label{equ1}
-\mathcal{L}_u[u_1] \leq 0\quad \hbox{ in } \Omega\, .
\end{equation}
Furthermore, since $u$ is zero on $\partial \Omega$,   changes sign in $\Omega$ and   is symmetric in the $x_1$--variable, we deduce that $u_1$ must change sign in $\Omega(e_1)$. Then, by the previous consideration on the sign of $u_1$ on $\partial \Omega \cap \partial \Omega(e_1)$, we conclude that there exists an open connected domain $D\subset \Omega(e_1)$ such that $u_1>0$ in $D$ and $u_1=0$ on $\partial D$. Thus, if by contradiction $\la_1^+ (\mathcal{L}_u, \Omega(e_1))\geq 0$, then $\la_1^+(\mathcal{L}_u, D)>0$ and, by \refe{equ1},   the maximum principle would imply the contradiction $u_1\leq 0$ in $D$.

\hfill$\Box$

When $\Omega=B$ is a ball, $F$ satisfies \refe{ri} and $u$ is a radial sign changing solution of \refe{eq2} with a finite number of nodal regions, then the assumption $\mathcal{N}(u)\cap \partial B=\emptyset$ is obviously satisfied. Hence, we can apply Proposition \ref{prop2} for any direction $e\in S^{n-1}$.

Analogously, if $\Omega=B$ is an annulus, though it is a domain not convex with respect to any direction, a proof similar to that of Proposition \ref{prop2} can be applied (see \cite{AP} for more details). Actually, we have the following result.
\begin{corollary}\label{rad-}If $B$ is a bounded radial domain, $F$ satisfies \refe{ue}, \refe{lip} and \refe{ri} and $u$ is a radial sign changing viscosity solution of \refe{eq2}, then
$$
\la_1^+ (\mathcal{L}_u, B(e))<0\quad \forall\, e\in S^{n-1}\, .
$$
\end{corollary}
\proof If $u$ has a finite number of nodal regions, then the conclusion follows directly from Proposition \ref{prop2} and the above considerations. If not, there exists a radial subdomain $\mathcal{B}\subset B$ in which $u$ has exactly two nodal regions. Hence, $\la^+_1(\lpu, B(e))<\la^+_1(\lpu, \mathcal{B}(e))<0$.

\hfill$\Box$

Finally, some extra considerations can be done for the special case of planar domains $\Omega\subset \R^2$ which are symmetric and convex with respect to two orthogonal directions, say $e_1=(1,0)$ and $e_2=(0,1)$. Note that this kind of domains need not to be convex, but they can be easily proved to be star--shaped with respect to the origin. 

Let us call doubly symmetric a continuous function $u$ which is symmetric with respect to  both directions $e_i$, $i=1,2$, i.e. a continuous function $u$ which is even in the variables $x_1$ and $x_2$. For such functions we have the following result.
\begin{lemma}\label{ds} Let $\Omega\subset \R^2$ be a domain  symmetric and convex  with respect to $e_i$, $i=1,2$, and let $u\in C(\overline \Omega)$ be a  sign changing, doubly symmetric function with two nodal regions. Then, $\mathcal{N}(u)\cap \partial \Omega=\emptyset$ and $0\notin \mathcal{N}(u)$, that is the nodal line of $u$ neither touches $\partial \Omega$ nor passes through the origin.
\end{lemma}
\proof Let us define $\Omega^+=\{ x\in \Omega\, :\, u(x)>0\}$ and $\Omega^-=\{ x\in \Omega\, :\, u(x)<0\}$.  By assumption, both $\Omega^\pm$ are connected open sets, hence  connected by arcs, and  symmetric with respect to $H(e_i)$, $i=1,2$. 

Let us consider a point $P_1\in \Omega^+\setminus \left( H(e_1)\cup H(e_2)\right)$ and let $P_2, P_3, P_4\in \Omega^+$ be the reflected points of $P_1$ with respect to $H(e_1)$, $H(e_2)$ and to the origin. Then, there exists a simple, closed curve $\gamma^+$ joining $P_1$, $P_2$, $P_3$ and $P_4$ and contained in $\Omega^+$, so that $u>0$ on $\gamma^+$. Obviously we can choose $\gamma^+$  not passing through the origin. By the Jordan curve theorem, $\R^2\setminus \gamma^+$ has two connected components, which we call $D_1$ and $D_2$, $D_1$ being the connected component containing the origin  and $D_2$ the one which contains $\partial \Omega$. Since $u$ has only two nodal regions, it follows that either $\Omega^-\subset D_1\cap \Omega$ or $\Omega^-\subset D_2\cap \Omega$. In the former case we immediately deduce that $\mathcal{N}(u)\cap \partial \Omega=\emptyset$. In the latter case, we can repeat the above construction in $\Omega^-$, that is we take in $\Omega^-$ four distinct symmetric points $Q_i$, $i=1,2,3,4$ as before, and select a simple closed curve $\gamma^-\subset \Omega^-$ passing through them. Again the Jordan curve theorem implies that $\R^2\setminus \gamma^-$ has two connected components, say $A_1$ which contains $\partial \Omega$, and $A_2$ which contains both $\gamma^+$ and the origin. Then, $\Omega^+$ must be contained in $A_2$, so that $u$ is negative in a neighborhood of $\partial \Omega$ and, again, $\mathcal{N}(u)\cap \partial \Omega=\emptyset$. A similar argument shows also that $0\notin \mathcal{N}(u)$.

\hfill$\Box$

As a  consequence of the previous lemma and Proposition \ref{prop2} we deduce the following
\begin{corollary}\label{ds1}
Assume that $u\in C^1(\overline \Omega)$ is a viscosity solution of \refe{eq2}, with $u$ and $\Omega$ as in Lemma \ref{ds} and $F$ satisfying \refe{ue}, \refe{lip} and \refe{sime} for $e=e_i$, $i=1,2$. Then
$$
\la_1^+(\mathcal{L}_u, \Omega(\pm e_i))<0\quad \hbox{ for } i=1,2\, .
$$
\end{corollary}

\section{Foliated Schwarz symmetry for viscosity solutions}\label{sec2}

The aim of this section is to establish either full radial  symmetry or partial symmetry properties, such as foliated Schwarz symmetry, for  viscosity solutions of fully nonlinear elliptic equations in bounded radial domains. Thus, we focus on solutions of the problem
\begin{equation}\label{eq3}
\left\{\begin{array}{c}
-F(x, D^2u)= f(|x|, u)\quad \hbox{ in } B\, ,\\[1ex]
u=0 \quad \hbox{ on } \partial B\, ,
\end{array}
\right.
\end{equation}
and the operator $F$ will be always assumed to satisfy \refe{ue}, \refe{lip} and \refe{ri}. 

As a first result, which easily follows from Lemma \ref{diffe} (ii), let us prove   the radial symmetry  of the usually called "stable" solutions.

\begin{theorem}\label{stable} Let $u$ be a viscosity solution of problem \refe{eq3} such that   $\lambda_1^+(\mathcal{L}_u, B)\geq 0$. Then, $u$ is radially symmetric in $B$.
\end{theorem}
\proof Let us fix any pair of linearly independent unit vectors $e, e'\in S^{n-1}$, and let us set $\theta=\theta_{e,e'}$. Then, by Lemma \ref{diffe} (ii) and the boundary condition in \refe{eq3}, the derivative $u_\theta$ satisfies, in the viscosity sense,
$$
\left\{\begin{array}{c}
-\mathcal{L}_u[u_\theta] \leq 0\quad \hbox{ in } B\\[1ex]
u_\theta=0 \quad \hbox{ on } \partial B
\end{array}\right.
$$
The assumption  $\lambda_1^+(\mathcal{L}_u, B)\geq 0$ implies that either $u_\theta\leq 0$ in $B$, or, by Proposition \refe{key}, $u_\theta>0$ in $B$. Moreover, in the first case, by the strong maximum principle, either $u_\theta<0$ in $B$ or $u_\theta \equiv 0$ in $B$. Therefore, three are the possible cases: $u_\theta<0$, $u_\theta>0$ or $u_\theta\equiv 0$ in $B$. But since $u$ is $2\pi$--periodic with respect to $\theta$, its derivative $u_\theta$ has to vanish somewhere in $B$. Hence, $u_\theta\equiv 0$ in $B$, and the arbitrariness of $e$ and $e'$ implies that $u$ is radially symmetric.

\hfill$\Box$

The definition of foliated Schwarz symmetric functions was recalled in the Introduction, see Definition \ref{fss}. Let us now
give some characterizations.

\begin{lemma}\label{fssc0} A function $u\in C(\overline{B})$ is foliated Schwarz symmetric if and only if for every $e\in S^{n-1}$ one has either $u(x)\geq u(\sigma_e(x))$ in $B(e)$ or $u(x)\leq u(\sigma_e(x))$ in $B(e)$. More precisely, $u$ is foliated Schwarz symmetric with respect to the direction $p\in S^{n-1}$ if and only if $u(x)\geq u(\sigma_e(x))$ for all $x\in B(e)$ and for every $e\in S^{n-1}$ such that $e\cdot p\geq 0$.
\end{lemma}

 This property was first stated in \cite{Br} and for a detailed proof we refer to \cite{W}. A different proof for solutions of semilinear elliptic equations can be found in \cite{GPW} (see also \cite{PR}).

On the other hand, for differentiable functions, the foliated Schwarz symmetry can be characterized as a sign property of the derivative $u_{\theta_{e,e'}}$, for linearly independent unit vectors $e, e'\in S^{n-1}$. 

\begin{proposition}\label{fssc1} A function $u\in C^1(B)\cap C(\overline{B})$ is foliated Schwarz symmetric if and only if there exists a direction $e\in S^{n-1}$ such that $u$ is symmetric with respect to $H(e)$ and for any other direction $e'\in S^{n-1}\setminus \{\pm e\}$ one has either $u_{ \theta_{e,e'}} \geq 0$ in $B(e)$ or  $u_{\theta_{e,e'}}\leq 0$ in $B(e)$.
\end{proposition}

Let us recall that the sufficiency of this condition was already observed in \cite{DGP} and \cite{PW}, but let us  include the proof for the sake of completeness.

\proof Let $u\in C^1(B)\cap C(\overline{B})$ be foliated Schwarz symmetric with respect to a direction $p\in S^{n-1}$, and let us fix $e\in S^{n-1}$ such that $e\cdot p=0$. Then, $u$ is clearly symmetric with respect to $H(e)$. Moreover, let $e'\in S^{n-1}$, with $e'\neq \pm e$. In order to show that either $u_{\theta_{e,e'}}\geq 0$ or $u_{\theta_{e,e'}}\leq 0$ in $B(e)$, we can assume that $e'\cdot e=0$ and that $\theta_{e,e'}\in [-\pi,\pi]$ is the angle formed by $e'$ and the orthogonal projection of $x$ on the plane $\Pi(e,e')$. We claim that if $e'\cdot p\geq 0$ then $u_{\theta_{e,e'}}\leq 0$ in $B(e)$, whereas if $e'\cdot p\leq 0$ then $u_{\theta_{e,e'}}\geq 0$ in $B(e)$. 

Indeed, using cylindrical coordinates with respect to $\Pi(e,e')$, let $x=(r, \theta, \eta)$ and $x'=(r, \theta', \eta)$ be in $B(e)$, for some $r>0$, $\eta\in \R^{n-2}$ and $0<\theta \leq \theta'<\pi$. Then, there exists $\nu\in S^{n-1}\cap \Pi(e,e')$ such that $x'=\sigma_\nu(x)$ ,  $\nu\cdot x>0$ and $\nu\cdot p>0$ if $p\cdot e'>0$, whereas $\nu\cdot p<0$ if $p\cdot e'<0$. Hence, by Lemma \ref{fssc0}, one has $u(x)\geq u(x')$ and therefore  $u_{\theta_{e,e'}}\leq 0$ in $B(e)$ provided that $p\cdot e'\geq 0$, as well as $u_{\theta_{e,e'}}\geq 0$ in $B(e)$ if $p\cdot e'\leq 0$. 

Conversely, assume that there exists $e\in S^{n-1}$such that $u$ is symmetric with respect to $H(e)$, and for every $e'\in S^{n-1}\setminus \{ \pm e\}$ the derivative $u_{\theta_{e,e'}}$ does not change sign in $B(e)$. Up to a rotation, we can assume that $e=e_2=(0,1,\ldots ,0)$.  Again by Lemma \ref{fssc0}, we have to prove  that for any $e'\in S^{n-1}$  either $u(x)\geq u(\sigma_{e'}(x))$ or $u(x)\leq u(\sigma_{e'}(x))$ for all $x\in B(e')$. By assumption if $e'=\pm e_2$ then $u(x)=u(\sigma_{e'}(x))$ , so we can assume $e'\neq \pm e_2$. Moreover, again up to a rotation around the $e_2$-axis, we can suppose that $e'$ lays on the $x_1x_2$-plane, with $e'=(\cos \theta_0, \sin \theta_0,\ldots ,0)$ for some $\theta_0\in \left( -\frac{\pi}{2}, \frac{\pi}{2}\right)$. Thus, the plane $\Pi (e, e')$ coincides with the $x_1x_2$-plane and let us denote just with $\theta\in [-\pi,\pi]$ the polar angle coordinate $\theta_{e,e'}$ given by the angle formed by the projection on the $x_1x_2$-plane with $e_1$. By assumption, we have that either $u_\theta (x)\geq 0$ or $u_\theta (x)\leq 0$ for all $x\in B$ with $x_2>0$. Moreover, using polar coordinates in the $x_1x_2$-plane,  the reflection map $\sigma_{e'}$ may be written as
$$
\sigma_{e'}(r\cos \theta, r\sin \theta, \tilde{x})= (r\cos (2\theta_0-\theta+\pi), r\sin (2\theta_0-\theta+\pi), \tilde{x})\, ,
$$
with $\tilde{x}=(x_3,\ldots,x_n)\in \R^{n-2}$. 

Now, we claim that $u(x)\leq u(\sigma_{e'}(x))$ in $B(e')$ if $u_\theta \geq 0$ in $B(e_2)$, whereas $u(x)\geq u(\sigma_{e'}(x))$ in $B(e')$ provided that $u_\theta \leq 0$ in $B(e_2)$. Indeed, assume that $u_\theta \geq 0$ in $B(e_2)$, and let us take first $x\in B(e')\cap B(e_2)$. Thus, we have $x=(r\cos \theta, r\sin \theta, \tilde{x})$ for some $\theta \in (0,\pi)$ such that $|\theta-\theta_0|<\frac{\pi}{2}$. This implies that the angle coordinate of the reflected point satisfies $2\theta_0-\theta+\pi>\theta>0$. Two cases are possible: either $\theta\geq 2\theta_0$ or $\theta<2\theta_0$. In the first case, both $x$ and $\sigma_{e'}(x)$ belong to $\overline{B(e_2)}$ and, by the non decreasing monotonicity with respect to $\theta$, it follows that $u(x)\leq u(\sigma_{e'}(x))$. In the latter case, by  periodicity, symmetry and  monotonicity, we again obtain 
$$
\begin{array}{ll}
u(\sigma_{e'}(x)) & =  u(r\cos(2\theta_0-\theta+\pi), r\sin (2\theta_0-\theta+\pi), \tilde{x})\\[2ex]
 & = u(r\cos(2\theta_0-\theta-\pi), r\sin (2\theta_0-\theta-\pi), \tilde{x})\\[2ex]
& = u(r\cos(-2\theta_0+\theta+\pi), r\sin (-2\theta_0+\theta+\pi), \tilde{x})\geq u(x)\, .
\end{array}
$$
Assume now that $x\in B(e')\setminus B(e_2)$, so that $\theta \in (-\pi, 0]$. Observe that we also have $2\theta_0-\theta -\pi <\theta\leq 0$. Again, we distinguish two cases: either $\theta \leq 2\theta_0$ or $\theta>2\theta_0$. In the former case, since $u_\theta \leq 0$ in $B(-e_2)$ by symmetry and both $x$ and $\sigma_{e'}(x)$ belong to $\overline{B(-e_2)}$, we immediately obtain
$$
u(x)\leq u(r\cos(2\theta_0-\theta -\pi), r\sin (2\theta_0-\theta -\pi), \tilde{x}) =u(\sigma_{e'}(x))\, .
$$
On the other hand, if $\theta>2\theta_0$, by symmetry and monotonicity as before, we have as well
$$
\begin{array}{ll}
u(\sigma_{e'}(x)) & =  u(r\cos(2\theta_0-\theta+\pi), r\sin (2\theta_0-\theta+\pi), \tilde{x})\\[2ex]
 & = u(r\cos(-2\theta_0+\theta-\pi), r\sin (-2\theta_0+\theta-\pi), \tilde{x})\geq u(x)\, .
\end{array}
$$
Hence, the inequality $u(x)\leq u(\sigma_{e'}(x))$ is proved in all cases. The same arguments can be used to show that 
$u(x)\geq u(\sigma_{e'}(x))$ if $u_\theta\leq 0$ in $B(e_2)$ and this concludes the proof.

\hfill$\Box$

By Proposition \ref{fssc1} and Lemma \ref{diffe} (ii) we can easily deduce a first symmetry result for viscosity solutions of \refe{eq3}, which is  the fully nonlinear extension of an analogous result for semilinear elliptic equation, see Proposition 2.3 in \cite{PW}. 

 \begin{theorem}\label{simandlam}
 Let $u$ be a  viscosity solution of \refe{eq3} and assume that there exists a direction $e\in S^{n-1}$ such that $u$ is symmetric with respect to $H(e)$. If $\lambda_1^+(\mathcal{L}_u, B(e))\geq 0$, then $u$ is foliated Schwarz symmetric and if $\lambda_1^+(\mathcal{L}_u, B(e))> 0$, then $u$ is radially symmetric.
\end{theorem}
\proof Let us show that, for any $e'\in S^{n-1}\setminus \{ \pm e\}$, $u_{\theta_{e,e'}}$ does not change sign in $B(e)$.  By Lemma \ref{diffe} (ii), by the boundary condition in \refe{eq3} and by the assumption of symmetry of $u$ with respect to $H(e)$,  the function $u_{\theta_{e,e'}}$  satisfies
$$
\left\{
\begin{array}{c}
-\mathcal{L}_u[ u_{\theta_{e,e'}}] \leq 0\quad \hbox{ in } B(e)\\[1ex]
u_{\theta_{e,e'}}=0 \quad \hbox{ on } \partial B(e)
\end{array}\right.
$$
Now, if $\lambda^+_1(\mathcal{L}_u, B(e))>0$, then the maximum principle holds true for operator $\mathcal{L}_u$ and we deduce $u_{\theta_{e,e'}}\leq 0$ in $B(e)$. On the other hand, if $\lambda^+_1(\mathcal{L}_u, B(e))=0$, then either $u_{\theta_{e,e'}}\leq 0$ in $B(e)$ or, by Proposition \ref{key}, $u_{\theta_{e,e'}}>0$ in $B(e)$. In any case, $u_{\theta_{e,e'}}$ does not change sign in $B(e)$. Then,  by Proposition \ref{fssc1}, $u$ is foliated Schwarz symmetric.
Moreover, if $\lambda^+_1(\mathcal{L}_u, B(e))>0$, again by Lemma \ref{diffe} (ii) and the maximum principle, we  obtain also $-u_{\theta_{e,e'}}\leq 0$ in $B(e)$. Hence, $u_{\theta_{e,e'}}\equiv 0$ in $B(e)$, and by the arbitrariness of $e'$ it follows that $u$ is radially symmetric.
 
 \hfill$\Box$

We are now ready to prove Theorem \ref{lam0}, which states that the a priori symmetry assumption on $u$ in Theorem \ref{simandlam} can be dropped, provided that $f(|x|,\cdot)$ is convex in $\R$.

 {\it Proof of Theorem \ref{lam0}}. Let $e$ be the direction for which $\lambda_1^+(\mathcal{L}_u, B(e))\geq 0$ and
let us set $w_e(x)=u(x)-u(\sigma_e(x))$. If $w_e\equiv 0$, then $u$ is symmetric with respect to $H(e)$, and we reach the conclusion by Theorem \ref{simandlam}. Therefore, we assume in the following $w_e\not\equiv 0$. By Lemma \ref{diffe} (i)   $w_e$ in particular satisfies, in the viscosity sense,
$$
\left\{
\begin{array}{c}
-\mathcal{L}_u [w_e]\leq 0 \quad \hbox{ in } B(e)\\[1ex]
w_e=0\quad \hbox{ on } \partial B(e)
\end{array}\right.
$$
Now, if $\lambda_1^+(\mathcal{L}_u, B(e))>0$, by the maximum principle one has $w_e\leq 0$ in $B(e)$ and then, by the strong maximum principle, $w_e<0$ in $B(e)$. If $\lambda_1^+(\mathcal{L}_u, B(e))=0$, then, by Proposition \ref{key}, either $w_e\leq 0$, and then again $w_e<0$ in $B(e)$, or $w_e>0$ (and $\mathcal{L}_u[w_e]=0$ in $B(e)$). Thus, in any case, we have  two possibilities: either
 $w_e<0$  or $w_e>0$  in $B(e)$. 

Next, in order to prove that $u$ is foliated Schwarz symmetric, we cannot apply directly Proposition \ref{fssc1}, since we are not able to find a fixed vector $\tilde{e}\in S^{n-1}$ such that $u$ is symmetric with respect to $H(\tilde{e})$ and for any other $e'\in S^{n-1}$, the function $u_{\theta_{\tilde{e},e'}}$ does not change sign in $B(\tilde{e})$. On the other hand, we can
 repeat the argument of the  proof of Proposition \ref{fssc1}: for any $\nu\in S^{n-1}$, in order to show that either $u(x)\geq u(\sigma_\nu (x))$ or $u(x)\leq u(\sigma_\nu(x))$ in $B(\nu)$, it is enough to show that there exists $e'\in S^{n-1}\cap \Pi(e,\nu)$ such that
 $u$ is symmetric with respect to $H(e')$ and the derivative $u_{\theta_{e,e'}}$ (or $u_{\theta_{e',\nu}}$) does not change sign in $B(e')$. Thus, the proof will be completed if we show  that for any plane $\Pi$ through $e$, there exists $e'\in S^{n-1}\cap \Pi$ such that  $u$ is symmetric with respect to $H(e')$ and the derivative $u_{\theta_{e,e'}}$ does not change sign in $B(e')$.

We first consider the case $w_e<0$ in $B(e)$. Without loss of generality, we assume that $e=(0,1, \ldots ,0)$ and that $\Pi$ is the plane spanned by $(1,0, \ldots, 0)$ and $e$. For $\theta\geq 0$, let us consider the direction $e(\theta)=(\sin \theta,  \cos \theta, 0, \ldots, 0)\in \Pi$, so that $e(0)=e$. We apply the rotating plane method in order to find $\theta'\in (0,\pi)$ such that $u$ is symmetric with respect to $H(e')$ with $e'=e(\theta')$. We set
\begin{equation}\label{tetap}
\theta' \, :=\sup \{ \tilde{\theta} \in [0,\pi)\, :\, w_{e(\theta)}< 0\ \hbox{in }\ B(e(\theta))\, ,\ \forall\, \theta\in [0, \tilde{\theta}]\}\, .
\end{equation}
We notice that $\theta'$ is well defined since $w_e<0$ in $B(e)$, and, by continuity, $w_{e(\theta')}\leq 0$ in $B(e(\theta'))$. This implies $\theta'<\pi$, since $w_{e(\pi)}=w_{-e}=-w_e\circ \sigma_e>0$ in $B(-e)$. We claim that $w_{e(\theta')}\equiv 0$, i.e. $u$ is symmetric with respect to $H(e(\theta'))$. For, assume by contradiction that $w_{e(\theta')}\not\equiv 0$, so that, by the strong maximum principle, $w_{e(\theta')}<0$ in $B(e(\theta'))$. In this case, we can find $\epsilon>0$ small enough such that the inequality $w_{e(\theta)}< 0$ in $B(e(\theta))$ holds true for all $\theta \in [0, \theta'+\epsilon)$, and this contradicts the definition of $\theta'$. Indeed, for $\epsilon$ sufficiently small, we can select a compact set $K\subset \bigcap_{\theta'\leq \theta<\theta'+\epsilon} B(e(\theta))$ such that, for all $\theta\in [\theta',\theta'+\epsilon)$ the measure of the set $B(e(\theta))\setminus K$ is so small that the operator $\mathcal{L}_u$ satisfies the maximum principle in $B(e(\theta))\setminus K$. Moreover, by assumption, there exists $\eta>0$ such that $w_{e(\theta')}\leq -\eta$ in $K$ and then, for $\epsilon$ small enough,  we have $w_{e(\theta)}\leq -\eta/2$ in $K$ for all $\theta\in [\theta', \theta'+\epsilon)$. Thus, by the maximum principle and the strong maximum principle, we have $w_{e(\theta)}< 0$ in $B(e(\theta))\setminus K$, and then $w_{e(\theta)}< 0$ in $B(e(\theta))$ for all $\theta\in [0, \theta'+\epsilon)$, in contrast with the choice of $\theta'$. 

We further observe that, by Hopf's lemma, for all $\theta \in [0, \theta')$, one has
$$
\frac{\partial}{\partial e(\theta)}w_{e(\theta)}=2 \, Du\cdot e(\theta)<0\ \hbox{ on } H(e(\theta))\cap B\, ,
$$
$e(\theta)$ being the inner unit normal vector to $B(e(\theta))$ on $\partial B(e(\theta))\cap B$. This implies that, with respect to the cylindrical coordinates $x=(r\cos \theta, -r\sin \theta, x_3,\ldots, x_n)$, one has
$$
u_\theta = \left\{ \begin{array}{ll}
-r\, Du\cdot e(\theta)>0 & \hbox{ in }\ B(e')\setminus B(e)\\[1ex]
r\, Du\cdot e(\theta)<0 & \hbox{ in }\ \overline{B(e)}\setminus \overline{B(e')}\\[1ex]
\pm r\ Du\cdot e(\theta')=0 & \hbox{ in }\ H(e')\cap B
\end{array}\right.
$$
By using also Lemma \ref{diffe} (ii), it then follows that $u_\theta$ in particular satisfies
$$
\left\{ \begin{array}{c}
-\mathcal{L}_u[-u_\theta]\leq 0 \quad \hbox{ in } B(e)\cap B(e')\\[1ex]
-u_\theta \leq 0 \quad \hbox{ on } \partial\left( B(e)\cap B(e')\right)
\end{array}\right.
$$
Since $\lambda_1^+\left(\mathcal{L}_u, B(e)\cap B(e')\right)>\lambda_1^+(\mathcal{L}_u, B(e))\geq 0$, we can apply the maximum principle, and then the strong maximum principle, in order  to deduce $-u_\theta < 0$ in 
$B(e)\cap B(e')$. Summing up, we have proved that $u_\theta >0$ in $B(e')$, and this concludes the proof in the case $w_e<0$ in $B(e)$.

On the other hand, if $w_e>0$ in $B(e)$, one has $w_{-e}<0$ in $B(-e)$ and we can apply again the rotating plane method starting with $e(0)=-e$ and considering the directions $e(\theta)=(\sin \theta, -\cos \theta, \ldots, 0)$ for $\theta\geq 0$. 
By defining $\theta'\in (0,\pi)$ as in \refe{tetap}, we find a unit vector $e'=e(\theta')\in S^{n-1}$ such that $u$ is symmetric with respect to $H(e')$ and $w_{e(\theta)}<0$ in $B(e(\theta))$ for all $\theta\in [0, \theta')$. By means of Hopf's lemma as above, we also deduce that, again with respect to the cylindrical coordinates $x=(r\cos \theta,  -r\sin \theta, x_3,\ldots, x_n)$, one has
$$
u_\theta = \left\{ \begin{array}{ll}
-r\, Du\cdot e(\theta)>0 & \hbox{ in }\ B(e')\setminus B(-e)\\[1ex]
r\, Du\cdot e(\theta)<0 & \hbox{ in }\ \overline{B(-e)}\setminus \overline{B(e')}\\[1ex]
\pm r\, Du\cdot e(\theta')=0 & \hbox{ in }\ H(e')\cap B
\end{array}\right.
$$
Then, the maximum principle applied to $u_\theta$ in $B(e)\setminus \overline{B(e')}$ yields $u_\theta<0$ in $B(-e')$, so that, by symmetry, $u_\theta>0$ in $B(e')$.

\hfill$\Box$

\begin{remark} {\rm Let us observe that the only assumption that there exists a direction $e\in S^{n-1}$ such that $\lambda_1^+(\mathcal{L}_u, B(e))>0$, i.e. the positivity of the principal eigenvalue in just one cap $B(e)$, does not imply  the radial symmetry of $u$. This is somehow in contrast with the assertion of Theorem \ref{simandlam} in the case when $\lambda_1^+(\mathcal{L}_u, B(e))> 0$; however one should note that in Theorem \ref{simandlam} the symmetry of $u$ with respect to $H(e)$ was assumed. A counterexample in the case when the symmetry assumption is dropped can be obtained by considering the least--energy (positive) solution of the semilinear problem
$$
\left\{\begin{array}{c}
-\Delta u=u^p\quad \hbox{ in } A\\[1ex]
u=0 \quad \hbox{ on } \partial A
\end{array}\right.
$$
where $A$ is an annulus in $\R^n$, $n\geq3$ and $p<\frac{n+2}{n-2}$ is close to the critical exponent $\frac{n+2}{n-2}$. It has been shown in several papers that $u$ is foliated Schwarz but not radially symmetric. On the other hand,  it is easy to see that there are directions (indeed,  infinitely many!) $e\in S^{n-1}$ such that $\lambda_1^+(\Delta +p\, |u|^{p-1}, B(e))>0$ whereas $\lambda_1^+(\Delta +p\, |u|^{p-1}, B(-e))<0$ and, obviously, $H(e)$ is not a symmetry hyperplane for $u$ (see \cite{PR} for more details).}
\end{remark}

By Theorem \ref{lam0}, at least for  convex nonlinearities $f$, the condition  
$\lambda_1^+(\mathcal{L}_u, B(e))\geq 0$ for some $e\in S^{n-1}$ is sufficient for $u$ to be 
foliated Schwarz symmetric. Concerning necessary conditions, we have the following result.

\begin{theorem}\label{lam-} Assume that problem \refe{eq3} has a solution $u$  which is not radial but foliated Schwarz  symmetric with respect to $p\in S^{n-1}$. Then, for all $e\in S^{n-1}$ such that $e\cdot p=0$, one has
$$\lambda^-_1(\mathcal{L}_u, B(e))\geq 0\, .$$
\end{theorem}
\proof For $e\in S^{n-1}$ orthogonal to $p$, let us denote  by $\theta$  the polar angle coordinate $\theta_{p,e}$ defined as  the angle formed by $p$ and the orthogonal projection of $x$ in the plane $\Pi(p,e)$. By Proposition \ref{fssc1} and by Lemma \ref{diffe} (ii), $u_\theta$ satisfies
$$
\left\{ \begin{array}{c}
-\mathcal{L}_u[u_\theta]\leq 0\, ,\ \quad \hbox{ in } B(e)\\[1ex]
u_\theta\leq 0 \ \hbox{ in } B(e)\, ,\ u_\theta=0 \quad \hbox{ on } \partial B(e)
\end{array}\right.
$$
The strong maximum principle implies that either $u_\theta <0$ or $u_\theta\equiv 0$ in $B(e)$. Since $u$ is not radially symmetric, we deduce $u_\theta<0$ in $B(e)$ and therefore, by its very definition, $\lambda_1^-(\mathcal{L}_u, B(e))\geq 0$. 

\hfill$\Box$

\begin{remark}\label{necelam+}
{\rm We notice that, if $u$ is not radial but foliated Schwarz symmetric with respect to $p$, then, for any $e\in S^{n-1}$ such that $e\cdot p=0$, we have  $\lambda_1^+(\mathcal{L}_u, B(e))\leq 0$ by Theorem \ref{simandlam}. Thus,  in the semilinear case for which $\alpha=\beta$ and $\lambda_1^+(\mathcal{L}_u, B(e))=\lambda_1^-(\mathcal{L}_u, B(e))$, Theorem \ref{lam-} yields that if $u$ is a not radial foliated Schwarz symmetric solution, then necessarily 
$$
\lambda_1(\Delta +f'(|x|,u), B(e))=0
$$
for all $e\in S^{n-1}$ orthogonal to the symmetry axes of $u$.}
\end{remark}

\section{Applications and spectral properties.}\label{sec4}
The main symmetry result of Theorem  \ref{lam0} was based on the assumption that there exists some direction $e\in S^{n-1}$ such that $\lambda_1^+(\lpu,B(e))\geq 0$. We wish to comment on this eigenvalue and  its role in providing bounds for the eigenvalues of the  operator $\Mp$.

Let us start by recalling that in the introduction, we introduced  the value
$\mu_2^+=\mu^+_2(\lpu,\Omega)$ defined in \refe{mu} for any bounded domain, and we showed that its
non negativity, when $\Omega=B$ is a radial domain, easily  implies, by Theorem \ref{lam0},  that $u$ is foliated Schwarz symmetric. It would be very interesting to study the sign of $\mu^+_2$ for positive solutions of \refe{eq3}, in particular for those found in \cite{QS}.

Let us observe that when $\alpha=\beta$, i.e. when $F$ is the Laplace operator, $\mu^+_2$ is the second  eigenvalue of $\lpu$, hence the inequality $\mu_2^+\geq 0$ just means that $u$ has Morse index less than or equal to one. On the contrary,
in the fully nonlinear case  the following proposition holds.
\begin{proposition}\label{secn}
If $\alpha<\beta$,  then $\mu^+_2(\lpu,\Omega)$ is not an eigenvalue for $\lpu$ in $\Omega$ with  corresponding sign changing eigenfunctions having exactly two nodal regions.
\end{proposition}
\begin{proof} 
Suppose by contradiction that $\mu_2^+=\mu^+_2(\lpu,\Omega)$ is such an eigenvalue.
Hence,    there exists  a sign changing function $\psi$ solution of
$$\left\{\begin{array}{c}
\Mp(D^2 \psi)+(f^\prime(|x|,u)+\mu^+_2)\psi=0 \quad \mbox{in}\ \Omega\\[1ex]
\psi=0\quad \mbox{on}\ \partial \Omega\, ,
\end{array}
\right.
$$
such that  $\Omega^{-}=\{x\in \Omega\, :\,  \psi(x)<0\}$ and $\Omega^{+}=\{x\in \Omega\, :\,  \psi(x)>0\}=\Omega\setminus \overline{\Omega^-}$ are subdomains of $\Omega$. Since $\alpha<\beta$, Proposition \ref{prl}  yields
$$\lambda_1^+(\lpu,\Omega^-)<\lambda_1^-(\lpu,\Omega^-)=\mu^+_2=\lambda_1^+(\lpu,\Omega^+).$$
By these inequalities and using again Proposition \ref{prl}, one can choose $D$ containing $\Omega^+$ but sufficiently close to it so that 
$$\lambda_1^+(\lpu,\Omega^-)<\lambda_1^+(\lpu,\Omega\setminus \overline D)\leq\lambda_1^+(\lpu,D)<\lambda_1^+(\lpu,\Omega^+)=\mu^+_2\, ,$$
and this contradicts the fact that by the definition (\ref{mu}) we  have $\mu^+_2\leq \lambda_1^+(\lpu,D)$.
\end{proof}

\begin{remark} {\rm The proof of Proposition \ref{secn} leads to believe that  a natural candidate for being the second eigenvalue of $\lpu$  should be
$$
\gamma^+_2(\lpu,B)=\inf_{D\subset B}\max\left\{\lambda_1^+(\lpu,D),\lambda_1^-(\lpu,B\setminus \overline D)\right\}\geq \mu_2^+(\lpu, B)\, .
$$
It would be also interesting to know whether  the non negativity of $\gamma^+_2(\lpu,B)$ would imply that $u$ is foliated Schwarz symmetric.}
\end{remark}
\smallskip

Let us remark that Proposition \ref{secn} holds for any operator of the form $\mathcal{L}[\varphi]=\Mp(D^2\varphi)+c(x)\varphi$. For simplicity we will, from now on, suppose that $c(x)=0$, i.e. we concentrate on the Pucci operator $\Mp$.

Let us   now state a few results related to eigenvalues higher than the principal ones that can be deduced as consequences of the symmetry result of Theorem \ref{lam0}. We wish to emphasise that in \cite{A}, Armstrong  defined 
$$\Lambda_2=\inf\{\lambda>\lambda_1^- \, :\, \lambda\ \mbox{ is an eigenvalue of $\Mp$}\ \}.$$
He then proved that $\Lambda_2>\lambda_1^-$ and  that for any $\mu\in (\lambda_1^-,\Lambda_2)$ and for any 
continuous $f$  there exists a solution of the Dirichlet problem
$$\left\{\begin{array}{lc}
\Mp(D^2 u)+\mu u=f & \mbox{in}\ B\\
u=0& \mbox{on}\ \partial B.
\end{array}
\right.
$$
Hence the importance of any estimate on $\Lambda_2$.

Let us  call \emph{nodal eigenvalues}  the eigenvalues that are not the principal ones,  since $\lambda_1^+$ and $\lambda_1^-$ are the only ones having  eigenfunctions that do not change sign.

For simplicity,  let us also denote by $\tilde{\lambda}_1^+=\lambda_1^+(\Mp,B(e))$, i.e. the principal eigenvalue in any half domain $B(e)$ (since it clearly does not depend on $e$), and by $\lambda_2^r$ the smallest radial nodal eigenvalue in $B$. 

\begin{theorem}\label{lame}
The following inequalities hold
$$\la_2^r>\tilde{\lambda}_1^+\quad \hbox{and}\quad  \Lambda_2 \geq \tilde{\lambda}_1^+\, .$$
\end{theorem}
\proof Remark first  that Corollary \ref{rad-} implies that if $\la$ is any nodal radial eigenvalue, then, for any $e\in S^{n-1}$,
\begin{equation}\label{y}
\lambda_1^+(\Mp+\lambda, B(e))< 0\, .
\end{equation}
 But
$$\lambda_1^+(\Mp+\lambda, B(e))=\tilde{\la}^+_1-\la\, ,$$
so that the first inequality of the statement follows.

Next, in order to prove the second inequality, suppose by contradiction that for some 
$
\lambda<\tilde{\lambda}_1^+
$
there exists $\psi\neq0$ sign changing solution of
$$\left\{\begin{array}{lc}
\Mp(D^2\psi)+\lambda \psi=0 & \mbox{in}\ B\\
\psi=0 & \mbox{on}\ \partial B.
\end{array}
\right.
$$
Then
$$\lambda_1^+(\mathcal{L}_\psi, B(e))=\lambda_1^+(\Mp+\lambda, B(e)))=\tilde{\la}^+_1-\la>0\, .
$$
By Proposition  \ref{prop1} it follows that $\psi$ is radially symmetric and then (\ref{y}) holds true, a contradiction. 

\hfill$\Box$

Let us observe that if it happens that $\lambda_1^-$ is larger than $\tilde{\lambda}_1^+$, then the estimate $\Lambda_2\geq \tilde{\lambda}_1^+$ provided by Theorem \ref{lame} is not relevant. However, when the ellipticity constants $\alpha$ and $\beta$ are sufficiently close to each other, this is not the case. It would be interesting to estimate the gap $\la_1^--\la_1^+$ in dependence of $\alpha$ and $\beta$ and the relation between $\la_1^-$ and $\tilde{\lambda}_1^+$.

In the two dimensional case, Theorem \ref{lame} can be extended to a larger class of domains, precisely to domains $\Omega$ which are symmetric and convex with respect to two orthogonal directions, say $e_1=(1,0)$ and $e_2=(0,1)$, i.e. the same kind of domains considered in Section 2.
Following the same notation,  we consider the eigenvalues $\la^+_1(\Mp,\Omega(e_1))$ and  
$\la^+_1(\Mp,\Omega(e_2))$.

By using Proposition \ref{prop1} and Corollary \ref{ds1}, the analogous result to Theorem \ref{lame} is
\begin{theorem} Let $\Omega$ be as in  Lemma \ref{ds} and 
let $\la$ be a nodal eigenvalue for $\Mp$ in $\Omega$ associated with an eigenfunction $\psi$ having two nodal regions.  Then:
\begin{itemize}
\item[(i)] $\la\geq \min\left\{ \la^+_1(\Mp,\Omega(e_1)), \la^+_1(\Mp,\Omega(e_2)) \right\}\, ;$
\smallskip

\item[(ii)] if $\psi$ is  doubly symmetric, then 
$$\la>\max\left\{ \la^+_1(\Mp,\Omega(e_1)), \la^+_1(\Mp,\Omega(e_2)) \right\}\, .$$
\end{itemize}
\end{theorem} 
The proof proceeds  as the one of Theorem \ref{lame}.

\medskip

To conclude, we observe that a question which remains open is whether $\tilde{\la}_1^+$ is a nodal eigenvalue for $\Mp$ in $B$,  as for the laplacian, or not. Note that if this was the case,  then, by Theorem \refe{lame},  $\tilde{\lambda}_1^+$ would be the smallest nodal eigenvalue of $\Mp$ in $B$.
 Next we  describe some qualitative  
properties that a  corresponding eigenfunction should have.

\begin{proposition}\label{pis}
Assume that $\tilde{\la}_1^+$ is a nodal eigenvalue for $\Mp$ in $B$ and that $\psi_2$ is a corresponding eigenfunction, i.e.
$$\left\{\begin{array}{c}
\Mp(D^2\psi_2)+\tilde{\la}_1^+ \psi_2=0 \quad \mbox{in}\ B\\[1ex]
\psi_2=0 \quad \mbox{on}\ \partial B
\end{array}
\right.
$$
Then
\begin{itemize}
\item[(i)] $\psi_2$ is not radial;
\item[(ii)] $\psi_2$ is foliated Schwarz symmetric;
\item[(iii)]  the nodal set of $\mathcal{N}(\psi_2)$ does intersect the boundary;
\item[(iv)] if $\alpha<\beta$, then, for any $e\in S^{n-1}$, $B^+:=\{x\in B\, :\,  \psi_2>0\}\neq B(e)$. 
\end{itemize}
\end{proposition}
\proof (i) is just the first inequality in Theorem \ref{lame}, and (ii) follows directly from  Theorem \ref{lam0}. Then, (ii) and  Proposition \ref{prop2} yield (iii).

Finally, in order to prove (iv), suppose by contradiction that, for some $e$, $B(e)=B^+$. This implies that $B^-=B(-e)$. Hence, $\la_1^-(\Mp,B(-e))=\la_1^+(\Mp,B(e))=\tilde{\la}^+_1$. On the other hand, the symmetry of the domain implies
$\la_1^-(B(-e))=\la_1^-(B(e))>\la_1^+(B(e))$, if $\alpha<\beta$. The contradiction proves the claim.

\hfill$\Box$

Let us denote by $\psi_1^+$ a positive eigenfunction in  $B(e)$ corresponding to $\tilde{\la}_1^+$. Then,  statement (iv) of Proposition \ref{pis} implies that $\psi_1$, the sign changing function constructed by odd reflection of $\psi_1^+$, is not an eigenfunction for $\Mp$  provided that $\alpha<\beta$, contrarily to the case when $\alpha=\beta$. The same argument shows that, if $\alpha<\beta$, then $\Mp$ cannot have a nodal eigenfunction antisymmetric with respect to $H(e)$ for some $e\in S^{n-1}$ and such that $B^+=B(e)$.

\end{document}